\documentclass[12pt, a4paper]{article}
\usepackage{titlesec} 
\usepackage{amsfonts}
\usepackage{amsthm}
\usepackage{amssymb}
\usepackage{amsmath}
\usepackage{theoremref}
\usepackage{enumerate}
\usepackage{bbm}
\usepackage{authblk}

\makeatletter
\g@addto@macro\bfseries{\boldmath}
\makeatother

\newcommand{\N}{\mathcal{N}}

\newcommand{\conj}[1]{\overline{#1}}

\newcommand{\Ip}[2]{\Big\langle #1, #2 \Big\rangle}

\newcommand{\supp}{\text{supp} \,}

\newcommand{\h}{\mathcal{H}}

\renewcommand\Im{\operatorname{Im}}

\newtheorem{thm}{Theorem}[section]
\newtheorem{lemma}[thm]{Lemma}
\newtheorem{cor}[thm]{Corollary}

\theoremstyle{definition}

\theoremstyle{definition}

\titleformat{\section}
{\bfseries\scshape\centering}{\thesection}{2em}{}

\titleformat{\subsection}
[runin]{\bfseries\scshape}{\thesubsection}{1em}{}

\titleformat{\subsubsection}
[runin]{\itshape}{\thesubsubsection}{0.5em}{}

\begin{document}
\title{\textbf{Nearly invariant subspaces of de Branges spaces}}
\author{Bartosz Malman}
\date{}
\maketitle

\begin{abstract} We prove that the nearly invariant subspaces of a de Branges space $\h(E)$ which have no common zeros are precisely of the form an exponential function times a de Branges space $\h(E_0)$.\end{abstract}

\section{Introduction and the main result}

Let $\h$ be a space of analytic functions defined on some domain of the complex plane $\mathbb{C}$. A concept commonly appearing in operator theory and complex function theory is that of nearly invariance of $\h$. The space is said to be nearly invariant if the zeros of functions in $\h$ can be divided out without leaving the space. More precisely, if $f \in \h$ and $f(\lambda) = 0$, then $f(z)/(z-\lambda) \in \h$. Nearly invariance is sometimes instead referred to as the division property. More generally, if all functions in the space $\h$ vanish on some subset of the complex plane, then the space $\h$ will be called nearly invariant if zeros of the functions in $\h$ can be divided out as long as they do not belong to the common zero set. 

This short article is concerned with the (always norm-closed) nearly invariant subspaces of the de Branges spaces. An entire function $E$ which satisfies the inequality $|E(z)| > |E(\conj{z})|$ for $z$ in the upper half plane $\mathbb{C}^+ = \{ z \in \mathbb{C} : \Im z > 0\}$ is called a de Branges function, and to each such function there exists an associated de Branges space $\h(E)$. Let $H^2(\mathbb{C}^+)$ denote the usual Hardy space of the upper half plane, and for an entire function $f$ define $f^*(z) := \conj{f(\conj{z})}$. The de Branges space $\h(E)$ is the Hilbert space of entire functions $f$ which satisfy the following three properties:
\begin{enumerate}[\itshape(i)]
\item $f/E \in H^2(\mathbb{C}^+)$,
\item $f^*/E \in H^2(\mathbb{C}^+)$, 
\item $\|f\|^2_{\h(E)} := \int_{\mathbb{R}} |f/E|^2 dx < \infty$.
\end{enumerate} The space $\h(E)$, with norm given by $(iii)$ above, is a reproducing kernel Hilbert space with kernel given by $$k_E(\lambda, z) = \frac{E(z)\conj{E(\lambda)} - E^*(z)\conj{E^*(\lambda)}}{2\pi i(\conj{\lambda} - z)}.$$ Conversely, any kernel of the above form, with $E$ a de Branges function, will of course be the reproducing kernel of a de Branges space. More background information on this class of spaces can be found in de Branges' monograph \cite{de1968hilbert}.

The result that we will be proving here is the following structure theorem for nearly invariant subspaces of de Branges spaces. 

\begin{thm} \thlabel{maintheorem}
Let $\N$ be a nearly invariant subspace with no common zeros of a de Branges space $\h(E)$. Then there exists a de Branges space $\h(E_0)$ and $\alpha \in \mathbb{R}$ such that $$\N = e^{i\alpha z}\h(E_0) = \{ e^{i\alpha z}f(z) : f \in \h(E_0)\}.$$
\end{thm} For some special classes of de Branges functions $E$ the above theorem can be refined, and as an application of \thref{maintheorem} we shall give a new proof of a result of \cite{aleman2008derivation} which characterizes the nearly invariant subspaces of the Paley-Wiener spaces. As usual, for $a > 0$, the Paley-Wiener space $PW_a$ consists of the entire functions $F$ which are the Fourier transforms $F = \hat{f}$ of functions $f$ in $L^2(-a,a)$, with an alternative characterization as the space of entire functions of exponential type at most $a$ which are square-integrable on the real axis. Equipped with the usual $L^2$-norm computed on the real axis, the space $PW_a$ is a de Branges space corresponding to the function $E(z) = e^{-ia z}$. 

\begin{cor}[\cite{aleman2008derivation}] \thlabel{PWnearinv} If $\N$ is a nearly invariant subspace with no common zeros of a Paley-Wiener space $PW_a$, then there exists an interval $I \subseteq (-a, a)$ such that $$\N = \{f \in PW_a : \supp \hat{f} \subset I \}.$$
\end{cor}

\section{Proofs}

We shall use the notation already introduced in the previous section. Furthermore, we shall use the concept of the Nevanlinna class of the upper half plane, which is the class of functions analytic in $\mathbb{C}^+$ that can be written as a quotient of two bounded analytic functions in $\mathbb{C}^+$. The lower half plane consisting of $z \in \mathbb{C}$ with negative imaginary part will be denoted by $\mathbb{C}^-$.

\begin{lemma} \thlabel{kernellemma} If $\N$ is a nearly invariant subspace with no common zeros of a de Branges space $\h(E)$, then the reproducing kernel $k_\N$ of $\N$ is of the form \begin{equation}
 k_\N(\lambda, z) = \frac{F(z)\conj{F(\lambda)} - G(z) \conj{G(\lambda)}}{i(z - \conj{\lambda})}, \label{kerneleq}
\end{equation} where the functions $F$ and $G$ are entire, and $F/E, F^*/E, G/E, G^*/E$ are in the Nevanlinna class of the upper half plane. 
\end{lemma}

\begin{proof}
The proof is very similar to the proof of \cite[Theorem 23]{de1968hilbert}. We reproduce the argument here for the reader's convenience. Fix $\alpha \in \mathbb{C} \setminus \mathbb{R}$. Let $f$ be a function in $\N$ such that $f(\conj{\alpha}) = 0$. For $\lambda \in \mathbb{C}$, the function $\frac{z-\conj{\alpha}}{z-\alpha}\big(k_\N(\lambda, z) - \frac{k_\N(\lambda, \alpha)}{k_\N(\alpha, \alpha)} k_\N(\alpha, z)\big)$ is in $\N$, and we have that 

\begin{gather*}
\Ip{f(z)}{\frac{z-\conj{\alpha}}{z-\alpha}\big(k_\N(\lambda, z) - \frac{k_\N(\lambda, \alpha)}{k_\N(\alpha, \alpha)} k_\N(\alpha, z)\big)}_{\h(E)} \\
= \Ip{\frac{z-\alpha}{z-\conj{\alpha}}f(z)}{k_\N(\lambda, z) - \frac{k_\N(\lambda, \alpha)}{k_\N(\alpha, \alpha)} k_\N(\alpha, z)}_{\h(E)} \\
= \frac{\lambda - \alpha}{\lambda - \conj{\alpha}}f(\lambda) = \Ip{f(z)}{\frac{\conj{\lambda}-\conj{\alpha}}{\conj{\lambda}-\alpha}k_\N(\lambda, z)}_{\h(E)}.
\end{gather*}
The function $$\frac{z-\conj{\alpha}}{z-\alpha}\big(k_\N(\lambda, z) - \frac{k_\N(\lambda, \alpha)}{k_\N(\alpha, \alpha)} k_\N(\alpha, z)\big) - \frac{\conj{\lambda}-\conj{\alpha}}{\conj{\lambda}-\alpha}k_\N(\lambda, z) \in \N$$ is thus orthogonal to any function in $\N$ which vanishes at $\conj{\alpha}$, and is thus a scalar multiple of $k_\N(\conj{\alpha}, z)$. Evaluation at $z = \conj{\alpha}$ shows that
\begin{gather*}
\frac{z-\conj{\alpha}}{z-\alpha}\Bigg(k_\N(\lambda, z) - \frac{k_\N(\lambda, \alpha)}{k_\N(\alpha, \alpha)} k_\N(\alpha, z)\Bigg) - \frac{\conj{\lambda}-\conj{\alpha}}{\conj{\lambda}-\alpha}k_\N(\lambda, z) \\ = -\frac{\conj{\lambda}-\conj{\alpha}}{\conj{\lambda}-\alpha}\cdot\frac{k_\N(\lambda,\conj{\alpha})}{k_\N(\conj{\alpha}, \conj{\alpha})}k_\N(\conj{\alpha}, z)
\end{gather*} from which we can solve for $k_\N(\lambda,z)$ to obtain that 
\begin{gather*}
k_\N(\lambda,z) = \frac{1}{(\conj{\alpha} - \alpha)(z - \conj{\lambda})}\Bigg(\frac{k_\N(\lambda, \alpha)}{k_\N(\alpha,\alpha)}(z-\conj{\alpha})(\conj{\lambda}-\alpha)k_\N(\alpha,z) \\ - \frac{k_\N(\conj{\lambda}, \conj{\alpha})}{k_\N(\conj{\alpha}, \conj{\alpha})}(z-\alpha)(\conj{\lambda} - \conj{\alpha})k_\N(\conj{\alpha}, z) \Bigg).
\end{gather*}\sloppy By setting $\alpha = -i/2$, $$F(z) = k_\N(\alpha,\alpha)^{-1/2}k_\N(\alpha, z)(z-\conj{\alpha})$$ and $$G(z) = k_\N(\conj{\alpha}, \conj{\alpha})^{-1/2}k_\N(\conj{\alpha}, z)(z-\alpha)$$ we see that the kernel $k_\N(\lambda,z)$ is of the form as suggested in \eqref{kerneleq}. Moreover, since $k_N(\alpha,z) \in \h(E)$, we have that $k_N(\alpha,z)/E(z) \in H^2(\mathbb{C}^+)$, and so it follows that $F/E$ is in the Nevanlinna class of the upper half plane. The same is clearly true for $F^*/E, G/E$ and $G^*/E$. 
\end{proof}

\begin{lemma} \thlabel{FGlemma} Let $F,G$ be the entire functions in the expression for the reproducing kernel of $\N$ in \eqref{kerneleq}. Then the following properties hold:

\begin{enumerate}[(i)]
\item $F^*F = G^*G$,
\item $|F(z)| > |G(z)|$ if $z \in \mathbb{C}^-$,
\item $|F(z)| < |G(z)|$ if $z \in \mathbb{C}^+$.
\end{enumerate}
\end{lemma}

\begin{proof}
\sloppy The function $k_\N(\lambda, z)$ is of course an entire function of $z$, and setting $z = \conj{\lambda}$ we see from \eqref{kerneleq} that $F(\conj{\lambda})\conj{F(\lambda)} - G(\conj{\lambda})\conj{G(\lambda)} = 0$. This establishes $(i)$. Properties $(ii)$ and $(iii)$ follow from setting $z = \lambda$ and the fact that $\N$ has no common zeros, so that $k_\N(\lambda,z)$ is not the zero function, and thus \begin{equation*}
0 < \|k_\N(\lambda, \cdot)\|^2_{\h(E)} = k_\N(\lambda, \lambda) = \frac{|F(\lambda)|^2 - |G(\lambda)|^2}{-2 \Im \lambda}.
\end{equation*}
\end{proof}

\begin{lemma} \thlabel{Uexplemma}
Let $F,G$ be the entire functions in the expression for the reproducing kernel of $\N$ in \eqref{kerneleq}, and set $U = G^*/F$. Then $U$ is an exponential function, i.e., $U(z) = e^{i\alpha z}$ for some $\alpha \in \mathbb{R}$. 
\end{lemma}

\begin{proof}
The lemma will be established by verifying a series of claims:

\begin{enumerate}[\itshape(a)]
\item $U$ has no zeros and no poles in $\mathbb{C}$, and is thus an entire function,
\item $|U(x)| = 1$ for every $x \in \mathbb{R}$,
\item $U$ is in the Nevanlinna class of the upper half plane.

\end{enumerate}

If the above three claims are established, then the fact that $U$ is an exponential function follows easily from the Nevanlinna factorization (see, for instance, \cite[Theorem 9 and Problem 27]{de1968hilbert}). We will now establish the three stated claims. Assume that $G^*(\lambda) = \conj{G(\conj{\lambda})} = 0$, and additionally that $\lambda$ does not lie on the real axis. We will show that $F$ has a zero at $\lambda$, of the same order as $G^*$. Since $|G(\conj{\lambda})| = 0$ we see from $(iii)$ of \thref{FGlemma} that $\conj{\lambda}$ must be in the lower half-plane. From $(i)$ of \thref{FGlemma} it follows that either $F$ or $F^*$ has a zero at $\lambda$, but from $(ii)$ of \thref{FGlemma} we see that $|F^*(\lambda)| = |F(\conj{\lambda})| > |G(\conj{\lambda})|$, so $F^*(\lambda)$ is non-zero. Consequently $F$ has a zero at $\lambda$, of the same order as $G^*$. Thus $U = G^*/F$ has no zeros outside of the real axis. In the same manner we can show that $U$ has no poles outside of the real axis.

If $x$ is on the real axis, then by $(iii)$ of \thref{FGlemma} we have $$|F(x)| = \lim_{y \to 0, y > 0} |F(x+iy)| \leq \lim_{ y \to 0, y > 0} |G(x+iy)| = |G(x)|.$$ By considering the limit  when $y < 0$ in a similar manner we obtain that $|F(x)| = |G(x)|$ for real $x$, so that $|U(x)| = 1$, and thus $U$ has no zeros (or poles) on the real axis. This complets the proof of claim $(a)$ and $(b)$. Claim $(c)$ follows from the fact that $U$ is a quotient of $G^*/E$ and $F/E$, which are in the Nevanlinna class  of the upper half plane by \thref{kernellemma}.
\end{proof}

\begin{lemma} \thlabel{UNdbs}
Let $\N$ be a nearly invariant subspace with no common zeros of a de Branges space $\h(E)$, with kernel given by \eqref{kerneleq}, $U(z) = e^{i\alpha z}$ as in \thref{Uexplemma}, and $\sqrt{U}(z) = e^{i\alpha z/2}$. Consider the space $$\h_0 = \sqrt{U}\N = \{ \sqrt{U}(z)f(z) : f(z) \in \N \}.$$ If $\h_0$ is normed by $$\| \sqrt{U}f \|_{\h_0} = \| f \|_{\h(E)},$$ then $\h_0$ is a de Branges space.
\end{lemma}

\begin{proof} It is easy to see from the definition of $\h_0$ and \thref{kernellemma} that the reproducing kernel $k_{\h_0}(\lambda, z)$ of the space is \begin{gather*} k_{\h_0}(\lambda, z) = \conj{\sqrt{U}(\lambda)}\sqrt{U}(z) k_\N(\lambda ,z) \\ = \frac{F(z)\sqrt{U}(z) \conj{F(\lambda)\sqrt{U}(\lambda)} - G(z)\sqrt{U}(z)\conj{G(\lambda)\sqrt{U}(\lambda)}}{i(z - \conj{\lambda})}. \end{gather*} Thus $\h_0$ will be a de Branges space if we can show that $(F\sqrt{U})^* = G\sqrt{U}$ and $|F(z)\sqrt{U}(z)| < |F(\conj{z})\sqrt{U}(\conj{z})|$ for $z \in \mathbb{C}^+$. The former follows easily from the equality $F^*F = G^*G$, which implies that $F^* = GU$. Indeed, we also have that $U^* = 1/U$, and thus \begin{gather*} (F\sqrt{U})^* = F^*/\sqrt{U} = G U/\sqrt{U} = G \sqrt{U}. \end{gather*} For the latter, note that $|F(\conj{z})\sqrt{U}(\conj{z})| = |F^*(z)\sqrt{U^*}(z)|$ and thus, by squaring, we need to establish the inequality in \begin{gather*}
|F(\conj{z})\sqrt{U}(\conj{z})|^2 = |F^*(z)|^2 |G(z)|/|F^*(z)| \\ = |F^*(z)G(z)| > |F(z)G^*(z)| \\ = |F(z)|^2|G^*(z)|/|F(z)|  = |F(z) \sqrt{U}(z)|^2.
\end{gather*} We see that this inequality above holds from parts $(i)$ and $(iii)$ of \thref{FGlemma}. Indeed, since $z \in \mathbb{C}^+$, we have $$|F(z)G^*(z)| < |G(z)G^*(z)| = |F(z)F^*(z)| < |G(z)F^*(z)|.$$
\end{proof}

\begin{proof}[Proof of \thref{maintheorem}] Follows now immediately from \thref{UNdbs}.
\end{proof}

\begin{proof}[Proof of \thref{PWnearinv}] By \thref{maintheorem} there exists $b \in \mathbb{R}$ such that $e^{ib z}\N$ is a de Branges space. Because $\N \subseteq PW_a$, it follows that $e^{ib z}\N$ is contained in $PW_{a + |b|}$. By the de Branges ordering theorem (see \cite[Theorem 35]{de1968hilbert}), for every $c \in (0,a + |b|)$, either $PW_c \subseteq e^{ib z}\N$ or $e^{ib z}\N \subseteq PW_c$. Let $c_0$ be the supremum of $c$ such that $PW_c \subseteq e^{ib z}\N$ and $c_1$ be the infimum of $c$ such that $e^{ib z}\N \subseteq PW_c$. If $c_0 < c_1$, then for any $c \in (c_0, c_1)$ we would have that $PW_c \not\subset e^{ib z}\N$ and $e^{ib z}\N \not\subset PW_c$, which contradicts the ordering theorem. Thus $c_0 = c_1$, and consequently $PW_{c_0} = e^{ib z}\N$, or $e^{-ib z}PW_{c_0} = \N$. This implies the validity of the statement of the corollary. \end{proof}

\bibliographystyle{siam}
\bibliography{mybib}

\begin{thebibliography}{1}

\bibitem{aleman2008derivation}
{\sc A.~Aleman and B.~Korenblum}, {\em Derivation-invariant subspaces of
  {$C^\infty$}}, Computational Methods and Function Theory, 8 (2008),
  pp.~493--512.

\bibitem{de1968hilbert}
{\sc L.~De~Branges}, {\em Hilbert spaces of entire functions}, Prentice-Hall,
  1968.

\end{thebibliography}

\end{document}